\documentclass[12pt]{amsart}
\usepackage[margin=1in]{geometry}
\usepackage{amsmath,amsfonts,amsthm,amssymb,bbm}
\usepackage{graphicx,color,dsfont}
\usepackage{enumitem}
\usepackage{fourier}

\newtheorem{theorem}{Theorem}
\newtheorem{lemma}{Lemma}
\newtheorem{proposition}{Proposition}

\newtheorem{definition}{Definition}
	
	\newtheorem{corollary}{Corollary}

\newtheorem{claim}{Claim}

 \theoremstyle{definition}
 
 \theoremstyle{remark}

 \numberwithin{equation}{section}

\newcommand{\vertiii}[1]{{\left\vert\kern-0.25ex\left\vert\kern-0.25ex\left\vert #1
    \right\vert\kern-0.25ex\right\vert\kern-0.25ex\right\vert}}

\renewcommand{\[}{\left[}

\newcommand{\f}[2]{\frac{#1}{#2}}

\newcommand{\cl}{{\mathcal L}}

\newcommand{\ga}{\gamma}

\newcommand{\la}{\lambda}

\newcommand{\La}{\Lambda}
\newcommand{\si}{\sigma}

\newcommand{\vp}{\varphi}

\newcommand{\om}{\omega}

\newcommand{\rone}{\mathbb R}

\newcommand{\dpr}[2]{\langle #1,#2 \rangle}

\newcommand{\cm}{\mathcal M}

\newcommand{\cd}{\mathcal D}

\newcommand{\ch}{\mathcal H}

\newcommand{\p}{\partial}

\newcommand{\beq}{\begin{equation}}
\newcommand{\eeq}{\end{equation}}
\newcommand{\beqna}{\begin{eqnarray*}}
\newcommand{\eeqna}{\end{eqnarray*}}
\newcommand{\beqn}{\begin{equation*}}
\newcommand{\eeqn}{\end{equation*}}
\newcommand{\bp}{\begin{proof}}
\newcommand{\ep}{\end{proof}}
\newcommand{\bprop}{\begin{proposition}}
\newcommand{\eprop}{\end{proposition}}
\newcommand{\bt}{\begin{theorem}}
\newcommand{\et}{\end{theorem}}
\newcommand{\bex}{\begin{Example}}
\newcommand{\eex}{\end{Example}}
\newcommand{\bc}{\begin{corollary}}
\newcommand{\ec}{\end{corollary}}
\newcommand{\bcl}{\begin{claim}}
\newcommand{\ecl}{\end{claim}}
\newcommand{\bl}{\begin{lemma}}
\newcommand{\el}{\end{lemma}}

\newcommand{\cj}{{\mathcal J}}
\newcommand{\ci}{{\mathcal I}}

\begin{document}



	\title{Spectral stability  of periodic waves for the
		Zakharov system}

	\author{Sevdzhan Hakkaev} 
	\address{Department of Mathematics, Faculty of Science, Trakya University, \\ 22030 Edirne, Turkey and \\ Institute of Mathematics and Informatics, \\ Bulgarian Academy of Sciences, Acad. G. Bonchev Str. bl. 8, 1113 Sofia, Bulgaria}
\email{s.hakkaev@shu.bg}
	
	\author[Milena Stanislavova]{\sc Milena Stanislavova}
	\address{Department of Mathematics, University of Alabama - Birmingham,
		1402 10th Avenue South
		Birmingham AL 35294, USA
	}
	\email{mstanisl@uab.edu}
\author[Atanas G. Stefanov]{\sc Atanas G. Stefanov}
\address{Department of Mathematics, University of Alabama - Birmingham,
	1402 10th Avenue South
	Birmingham AL 35294, USA
}
\email{stefanov@uab.edu}\footnote{Corresponding author}

	\date{}
	\maketitle

\date{\today}

\begin{abstract}
	The paper concerns with  the stability of periodic travelling waves of dnoidal type, 
	of the Zakharov system. This problem was considered in a work of J. Angulo and C. Brango, Nonlinearity {\bf 24}, 2913 (2011).  Specifically, it was shown that under  a technical condition on the perturbation, such waves are orbitally stable, with respect to perturbations of the same period.  Our main result fills up the gap created by the aforementioned technical condition. 
	More precisely, we show that  for all natural values of the parameters, the
periodic dnoidal waves are spectrally stable.

\end{abstract}

\maketitle

  \section{Introduction}
 We consider the  Zakharov system, which is the following system of coupled nonlinear PDE's 
  \begin{equation}\label{1.1}
  	\left\{ \begin{array}{ll}
  		v_{tt}-v_{xx}=\frac{1}{2}(|u |^2)_{xx} \\
  		\\
  		iu_t+u_{xx}-uv =0,
  	\end{array} \right.
  \end{equation}
 Specifically, $v$ is a real-valued  function, while  $u$ is a
  complex-valued function. The problem (\ref{1.1})  was introduced in
  \cite{Za} to describe Langmuir turbulence in a plasma.

  The problem of the stability of
  solitary waves for nonlinear dispersive equations goes back to the
  works of Benjamin \cite{Be} and Bona \cite{Bo} (see also
  \cite{AlBoHe, AlBo, We1, We2}). A general approach for
  investigating the stability of solitary waves for nonlinear
  equations having a group of symmetries was proposed in
  \cite{GrShSt1}.
  The well-posedness theory for Zakharov system in the periodic
  setting was investigated in \cite{GuSh}.
  In \cite{AB, Wu} the existence and stability of smooth solitary wave solutions is
  considered, in fact we state for reference purposes the precise stability results of \cite{AB} below.  

  The goal  of this paper is to consider the spectral stability of periodic travelling wave solutions of
  the form
  \begin{equation}\label{1.2}
  	\left\{
  	\begin{array}{ll}
  		v(t,x)=\psi(x-ct) \\
  		\\
  		u(t,x)=e^{-i\om
  			t}e^{i{\frac{c}{2}}(x-ct)}\phi (x-ct),
  	\end{array} \right.
  \end{equation}
  where $\psi, \;  \phi :
  \rone \rightarrow \rone$ are smooth,periodic
  functions with fixed period $2T$, and $\om, c \in \rone$. In order to ensure that the traveling wave $u$ above is $2T$ periodic, we will require that there is an integer $l$, so that
  \begin{equation}
  	\label{a:10}
  	c = \f{2 \pi l}{T}.
  \end{equation}
  We now construct such waves.
  \subsection{Construction of the periodic waves for the Zakharov system}
  Substituting (\ref{1.2}) in (\ref{1.1}) we obtain
  \begin{equation}\label{2.1}
  	\left\{ \begin{array}{ll}
  		(c^2-1)\psi''=\frac{1}{2}(\phi^2)''\\
  		\\
  		\phi''+\left( w+\frac{c^2}{4} \right)\phi=\phi\psi
  	\end{array} \right.
  \end{equation}
  Integrating the first equation in (\ref{2.1}), we get
   $$
  	\psi=-\frac{\phi^2}{2(1-c^2)}+a_0+b_0 x
 $$
 By the periodicity of $\phi, \psi$, we immediately conclude $b_0=0$. For the rest, we also consider $a_0=0$, as the other cases easily reduce, without loss of generality,  to this one by a simple change of parameters, see the defining equations \eqref{1.1}. That is,
  \begin{equation}
  	\label{2.3}
 	\psi=-\frac{\phi^2}{2(1-c^2)}
 \end{equation}
Using the relation \eqref{2.3} in the second equation of \eqref{2.1}, we get the following equation for $\phi$,
  \begin{equation}\label{2.4}
  	-\phi''+\sigma \phi-\frac{\phi^3}{2(1-c^2)}=0,
  \end{equation}
  where we have introduced the new parameter $\sigma:=-\om-\frac{c^2}{4}$.
 Multiplying by $\phi$ and integrating once, we get
  \begin{equation}\label{2.5}
  	\phi'^2=\frac{1}{4(1-c^2)}\left[ -\phi^4+4\sigma (1-c^2)\phi^2+a_1\right] ,
  \end{equation}
  where $a_1$ is a constant of integration. This is a Newton's equation, which is well-studied in the literature. In fact, one can construct several different type of solutions in terms of elliptic functions, including dnoidal, cnoidal and even snoidal solutions. Unfortunately, our preliminary  results for the cnoidal and snoidal type waves are far from definitive, so we will restrict our attention to the dnoidal waves. The cnoidal and snoidal waves will be a subject of a  future publication. 
  
  Next, we present the construction of the dnoidal waves.  
 Later, we  will state some relevant spectral properties of the corresponding linearized operator, as they will be essential for our considerations in the sequel.
  
  Let  $1-c^2>0$ and $\sigma>0$. Assume that the quadratic  equation
  $r^2 -4\sigma (1-c^2)r -a_1= 0$ has two positive roots $r_0>r_1>0$, and set $\phi_0=\sqrt{r_0} >\phi_1=\sqrt{r_1} > 0$. Clearly, there is an even and decreasing in $[0,T]$ peiodic solution of \eqref{2.5}, with 
  $$
  \phi(0)=\max_{0<x<T} \phi(x)=\phi_0, \phi(T)=\min_{0<x<L} \phi(x)=\phi_1.
  $$
   These are explicitly given, up to a translation,  as follows.
  \begin{proposition}(Existence of dnoidal solutions)
  	\label{prop:1}
  	
  	Let $1-c^2>0, \sigma>0$. Assume that the quadratic  equation
  	$r^2 -4\sigma (1-c^2)r -a_1= 0$ has two positive roots, denoted by $\phi_0^2>\phi_1^2$.  Then, the solution to \eqref{2.5} is given  by
  	\begin{equation}\label{4.1}
  		\phi(x)=\phi_0dn(\alpha x, \kappa),
  	\end{equation}
  	where
  	\begin{equation}\label{4.2}
  		\kappa^2=\frac{\phi_0^2-\phi_1^2}{\phi_0^2}=\frac{2\phi_0^2-4\sigma (1-c^2)}{\phi_0^2}, \; \; \alpha^2=\frac{1}{4(1-c^2)}\phi_0^2=\frac{\sigma}{2-\kappa^2}.
  	\end{equation}
  	In addition, the fundamental period of $\phi$ is
  	$$
  	2T=\frac{2K(k)}{\alpha}.
  	$$
  \end{proposition}

   We now turn our attention to the spectral stability of such solutions, in the context of the Zakharov system \eqref{1.1}.
  \subsection{The linearized problem}
   For the purposes of linearization, we rewrite
  the system (\ref{1.1}) as a first order in time system, in the form
  \begin{equation}\label{3.1}
  	\left\{  \begin{array}{ll}
  		v_t=-V_x , \;
  		\\
  		V_t=-(v+\frac{1}{2}|u |^2)_x
  		\\
  		iu_t+u_{xx}=uv.
  	\end{array} \right.
  \end{equation}
  Note that we enforce uniqueness by adding the condition $\int_{-T}^T{V(t,x)}dx=0$. Consider the perturbations in the form
  \begin{eqnarray*}
  	u(t,x) &=&e^{-iwt}e^{i\frac{c}{2}(x-ct)}[\phi(x-ct)+p(t,x-ct)]\\
  	v(t,x) &=& \psi(x-ct)+q(t,x-ct)\\
  	V(t,x)&=& \varphi(x-ct)+h(t,x-ct)
  \end{eqnarray*}
  where $q$ and $r$ are real-valued functions, and $p$ is complex-valued function. Here  $\vp$ may be identified  as the unique mean-value zero function, i.e. $\int_{-T}^T \vp(x) dx=0$, satisfying
  \begin{equation}
  	\label{c:10}
  	 \vp'=c\psi'=\f{1}{c}(\psi+\f{\phi^2}{2})'.
  \end{equation}
 Note that \eqref{c:10} is consistent with the zero order terms in \eqref{3.1} as well as \eqref{2.3}. Accordingly, as $V$ is mean value zero as well, we must require that the perturbation $h$ is mean-value zero as well, $\int_{-T}^T h(t,x) dx=0$.

  Plugging in (\ref{3.1}) and ignoring quadratic and higher order terms, we get the following linear system
  \begin{equation}\label{3.2}
  	\left\{ \begin{array}{ll}
  		q_t=cq_x-h_x\\
  		r_t=cr_x-q_x-(\phi \Re p)_x\\
  		ip_t=-p_{xx}-\left(w+\frac{c^2}{4}\right)p+\psi p+\phi q.
  	\end{array} \right.
  \end{equation}
  Further, by letting  $p=p_1+ip_2$, the system (\ref{3.2}) takes the form
  \begin{equation}\label{3.3}
  	\left\{ \begin{array}{ll}
  		q_t=cq_x-h_x\\
  		\\
  		r_t=ch_x-q_x-(\phi \Re p)_x\\
  		\\
  		p_{1t}=-p_{2xx}-\left(w+\frac{c^2}{4}\right)p_2+\psi p_2\\
  		\\
  		-p_{2t}=-p_{1xx}-\left(w+\frac{c^2}{4}\right)p_1+\psi p_1+\phi q.
  	\end{array} \right.
  \end{equation}
  For $\vec{U}=(p_2,p_1,q,h)$, the above system can be written in the form
  \begin{equation}\label{3.4}
  	\vec{U}_t=\mathcal{J}\mathcal{H}\vec{U},
  \end{equation}
  where
  \begin{eqnarray}
  	\label{3.5}
  	\mathcal{J} &=& \begin{pmatrix} 0&-1&0&0\\ 1&0&0&0\\ 0&0&0& -\partial_x\\ 0&0&-\partial_x&0 \end{pmatrix}, \; \; \mathcal{H}=\begin{pmatrix} \cl_-&0&0 &0\\  0&\cl_-&\phi &0\\ 0&\phi &1&-c\\ 0&0&-c&1 \end{pmatrix} \\
   \label{3.6}
  	\cl_-&=& -\partial_x^2+\sigma +\psi = -\partial_x^2+\sigma  - \f{\phi^2}{2(1-c^2)}.
  \end{eqnarray}
  Clearly $\cj^*=-\cj$, whereas $\ch^*=\ch$, where we associate to the operators $\cj, \ch$ the following  domains on the periodic functions
  \begin{eqnarray*}
  	 D(\cj)&=& (L^2[-T,T])^2\oplus (H^1[-T,T])^2\\
  	D(\ch) &=& (H^2[-T,T])^2\oplus L^2[-T,T]\oplus L^2_0[-T,T].
  \end{eqnarray*}
Note that the last component of the domain of $\ch$ is $L^2_0[-T,T]=\{f\in L^2[-T,T]: \int_{-T}^T f(x) dx=0\}$ per our earlier requirement that the last component $h\in L^2_0[-T,T]$.

  Transforming the time-dependent linearized problem \eqref{3.4} into an eigenvalue problem, through the transformation $\vec{U}\to e^{\la t} \vec{U}$, yields
  \begin{equation}\label{3.41}
  	\mathcal{J}\mathcal{H}\vec{U}=\la \vec{U}
  \end{equation}
  As we are in the periodic context, it is well-known that all essential spectrum is empty, thus reducing the spectrum to pure  point spectrum, that is isolated eigenvalues with finite multplicities.  The standard notion of stability is given next.
  \begin{definition}
  	\label{defi:10} We say that the traveling wave solution described in \eqref{1.2}, \eqref{2.3}, \eqref{2.5}
  	 is spectrally stable, if the eigenvalue problem \eqref{3.41} does not have non-trivial solutions with $\Re\la>0$. That is,
  	 $$
  	 (\la, \vec{U}): \Re \la>0, \vec{U} \neq 0,  \vec{U}\in D(\cj\ch) =  (H^2[-T,T])^2\oplus (H^1[-T,T]\oplus H^1_0[-T,T].
  	 $$
  	  Otherwise, if there are such solutions, we refer to  the familly in \eqref{1.2} as spectrally 
  	  unstable.
  	  
  	  For orbital stability, we refer to the standard formulation, see \eqref{b:8} and \eqref{b:12} below. 
  \end{definition}

  \subsection{Main result}
  The following is the main result of this work. 
  \begin{theorem}\label{thm1}
  Periodic traveling waves of dnoidal type of (\ref{1.1}) are spectraly stable for all natural values of the parameters. 
  
  More specifically, the periodic dnoidal waves constructed in Proposition \ref{prop:1}, with the respective speed,  subject to\footnote{recall that this condition is necessary to guarantee the periodicity of such waves},  \eqref{a:10} are spectrally stable solutions of \eqref{3.1}. 
  \end{theorem}
  It is worth to note, that in \cite{AB}, orbital stability of periodic waves of dnoidal type for the system (\ref{1.1} was proved. Angulo and Brango, \cite{AB} have shown  that for the equivalent system \eqref{3.1}, there is orbital stability, if one asks for an additional technical condition, see \eqref{ang:10} below.  More precisely, they proved that for all $\varepsilon>0$, there exists $\delta>0$ such that for any initial data $(v_0,V_0, u_0)\in L^2[-T,T]\times L_0^2[-T,T]\times H^1[-T,T]$ satisfying
\begin{equation}
	\label{b:8} 
	\|v_0-\psi\|_{L^2[-T,T]}<\delta, \; \; \|V_0-\varphi\|_{L^2[-T,T]}<\delta, \; \; 
	\|u_0-  \phi\|_{H^1[-T,T]}<\delta
\end{equation}
then
\begin{equation}
	\label{b:12}
	\begin{cases}
		\inf_{y\in \mathbf{R}}||v(\cdot+y,t)-\psi||_{L^2[-T,T]}<\varepsilon, \; \; \inf_{y\in \mathbf{R}}||V(\cdot+y,t)-\psi||_{L^2[-T,T]}<\varepsilon, \\ 
		\inf_{(\theta,y)\in [0,2\pi)\times\mathbf{R}}||e^{i\theta}u(\cdot+y,t)-\phi||_{H^1[-T,T]}<\varepsilon
	\end{cases}
\end{equation}
if 
\begin{equation}
	\label{ang:10} 
	\int_0^Tv_0(x)dx\leq \int_{0}^{T}\psi(x)dx.
\end{equation}
This result is established by adapting the results in \cite{Be, Bo, We1} to the periodic case.

Our work is structured as follows. In Section \ref{sec:2} below, we provide some basic and preliminary results - about the instability index counting theory and the relation of the linearized operators $\cl_\pm$ to the classical Schr\"odinger operators arising in the elliptic function theory. Next, in Section \ref{sec:3}, we develop the spectral theory for the self-adjoint matrix  linearized operator $\ch$, and the full linearized operator $\cj \ch$. In particular, we describe the kernels and the generalized kernels in full details. We also show that the Morse index of $\ch\leq 1$. This is later upgraded in Section \ref{sec:4} to $n(\ch)=1$. Section \ref{sec:4} also contains the proof of the main result, namely  the spectral  and orbital stability of the dnoidal waves. Spectral stability is achieved via the instability index count, while the orbital stability is obtained as a consequence of an abstract result, Theorem 5.2.11, p. 143, \cite{KP}, which relates the two notions. 

  \section{Preliminaries}
  \label{sec:2} 
  We start with some facts about the instability index counting theory for eigenvalue problems of the type described in \eqref{3.41}.
  \subsection{Instability index counting}
  \label{sec:2.1}
  We will give some  results about the instability index count theories developed in \cite{LZ}.  These allow us to count the number of unstable eigenvalues for eigenvalue problems of the form (\ref{s4.1}) (see below) based on the information about the the spectrum of various self-adjoint operators, both scalar
  and matrix, and some specific estimates.  For eigenvalue problem in the form
  \begin{equation}
  	\label{s4.1}
  	\mathcal{I} \mathcal{L} z=\lambda z.
  \end{equation}
Here, our standing assumption is  that for appropriate Hilbert space $X$, $\cl:X\to X^*$ is bounded and symmetric, and in addition  $\cl=\cl^*$ on appropriately defined Hilbert space $X\subset H \subset X^*$, and domain $D(\cl)$. In addition, assume that $\cl$  also a finite number of negative eigenvalues, $n(\mathcal{L})$, a quantity  referred to as Morse index of the operator $\mathcal{L}$.   In addition, $\mathcal{I}^*=-\mathcal{I}$.

  Let $k_r$  be  the sum of algebraic multiplicities of positive eigenvalues of the spectral problem \eqref{s4.1} (i.e. the number of real instabilities or real modes), $k_c$ be the sum of algebraic multiplicities of quadruplets of eigenvalues with non-zero real and imaginary parts, and $k_i^-$, the number of pairs of purely imaginary eigenvalues with negative Krein-signature.  For a simple pair of imaginary eigenvalues $\pm i \mu, \mu\neq 0$, and the corresponding eigenvector
  $\vec{z} = \left(\begin{array}{c}
  	z_1  \\ z_2
  \end{array}\right) $, the Krein signature, either $\pm 1$ is the following quantity
  $
  sgn(\dpr{ \mathcal{L}  \vec{z}}{ \vec{z}}).
  $

  Also of importance in this theory is a finite dimensional matrix $\cd$, which is obtained from the adjoint eigenvectors for \eqref{s4.1}. More specifically, consider  the generalized kernel of $\mathcal{I}\mathcal{L}$
  $$
  gKer(\mathcal{I}\mathcal{L})=span[(Ker[(\mathcal{I}\mathcal{L})^l], l=1, 2, \ldots]  .
  $$
  Assume that the dimension of the space $gKer(\ci\cl)\ominus Ker(\cl)$ is finite\footnote{ More generally, for subspaces $A\subset B$ of a fixed Banach space , one can always complement $A$ in $B$, if $dim(B\setminus A)<\infty$. That is, $B=A\oplus A_1$. In such a case, we denote $A_1:=B\ominus A$. Note that $A_1$ is not unique, so we denote by $B\setminus A$ any subspace with the property $A\oplus A_1=B$}. 
  Select a basis in 
  $$
  gKer(\ci\cl)\ominus Ker(\cl)=span[\eta_j, j=1, \ldots, N].
  $$
  Then $\cd\in \cm_{N\times N}$ is defined via
  $$
  \cd:=\{\cd_{i j}\}_{i,j=1}^N: \cd_{i j}=\dpr{\cl \eta_i}{\eta_j} .
  $$
  Then, according to  \cite{LZ}, we have the following formula, relating the number of
  ``instabilities'' or Hamiltonian  index of the eigenvalue problem \eqref{s4.1} and the Morse indices of $\cl$ and $\cd$
  \begin{equation}
  	\label{b:20}
  	k_{Ham}:=k_r+2 k_c+2k_i^-=n(\cl)-n_0(\cd),
  \end{equation}
 where $n_0(\cd)=\#\{\la\leq 0: \la\in \si(\cd)\}$ is the number of non-positive eigenvalues of $\cd$.

  {\bf Remark:} 
  As an easy corollary,  if $n(\mathcal{L})=1$, it follows from \eqref{b:20} that
  	$k_c=k_i^-=0$ and
  	\begin{equation}
  		\label{170}
  		k_r=1-n(\mathcal{D}).
  	\end{equation}
  	Thus, in the case $n(\cl)=1$, instability occurs exactly when $n(\cd)=0$, while stability occurs whenever $n(\cd)=1$.

\subsection{The linearized operators $\cl_\pm$ in terms of the standard Hill operators }
\label{sec:2.3}
We start by introducing another classical linearized Schr\"odinger operator, associated with the wave $\phi$, namely
$$
\cl_+:=-\p_x^2 + \si - \f{3}{2(1-c^2)} \phi^2.
$$
We now consider two  concrete classical Hill operators, which are related to the linearized operators $\cl_\p,$ along with some relevant spectral properties. These will allow us to accurately determine the negative spectrum and the kernel of the scalar Schr\"odinger operators, which will be of use in the sequel.

More specifically, the Schr\"odinger operator
$$
\Lambda_1=-\partial_{y}^{2}+6k^{2}sn^{2}(y, k),
$$
with periodic boundary conditions on $[0, 4K(k)]$, has  eigenvalues  that are all
simple. The first few  eigenvalues and corresponding eigenfunctions are given in the list below
$$\begin{array}{ll}
	\nu_{0}=2+2k^2-2\sqrt{1-k^2+k^4}; \ \  \phi_{0}(y)=1-(1+k^2-\sqrt{1-k^{2}
		+k^{4}})sn^{2}(y, k),\\ \\
	\nu_{1}=1+k^{2}, \ \
	\phi_{1}(y)=cn(y, k)dn(y, k)
	=sn'(y, k),\\ \\
	\nu_{2}=1+4k^{2}, \ \
	\phi_{2}(y)=sn(y, k)dn(y, k)
	=-cn'(y, k),\\
	\\
	\nu_{3}=4+k^{2}, \ \
	\phi_{1}(y)=sn(y, k)cn(y, k)
	=-k^{-2}dn'(y, k),\\[1mm]
\end{array}
$$
Similarly, for the operator
$$
\Lambda_2=-\partial_y^2+2\kappa^2sn^2(y,
\kappa),
$$
 with  periodic boundary conditions on $[0, 4K(k)]$, eigenvalues   are all
simple. The first three  eigenvalues and corresponding eigenfunctions are:
$$
\left\{\begin{array}{ll}
	\epsilon_0=k^2, & \theta_0(y)=dn(y, k),\\[1mm]
	\epsilon_1=1, & \theta_1(y)=cn(y, k),\\[1mm]
	\epsilon_2=1+k^2, & \theta_2(y)=sn(y, k).
\end{array}
\right.
$$

We now relate the Schr\"odinger operators $\cl_\pm$ with $\La_1, \La_2$. We start with the dnoidal case.
\subsubsection{The operators $\cl_+, \cl_-$ in terms of $\La_1, \La_2$}

An elementary and classical calculation shows 
\begin{eqnarray}
	\label{4.3}
	\cl_+ &=&\alpha^2[\Lambda_1-(4+\kappa^2)]\\
	\label{4.4}
	\cl_- &=& \alpha^2[\Lambda_2-\kappa^2].
\end{eqnarray}
Based on the formula \eqref{4.1} and \eqref{4.2}, we can formulate the following useful spectral properties.
\begin{proposition}
	\label{prop:5}
	The linearized operators $\cl_\pm$ have the following spectral properties:
	\begin{itemize}
		\item $n(\cl_+)=1$, $\ker(\cl_+)=span[\phi']$
		\item $\cl_-\geq 0$, $\ker(\cl_-)=span[\phi]$.
	\end{itemize}
\end{proposition}

   \section{Spectral theory for $\ch$ and $\cj \ch $}
   \label{sec:3} 
   
We now discuss some elementary spectral properties of the scalar Schr\"odinger  operators $\cl_\pm$.
\subsection{Elementary properties of $\cl_\pm$}
Note that $\cl_-$ and $\cl_+$ have some generic properties, which can be gleaned directly from the defining equation \eqref{2.4}. More precisely, we have the following two relations
\begin{equation}
	\label{a:23}
	\cl_-[\phi]=0; \ \ \cl_+[\phi']=0.
\end{equation}
Indeed, the formula  $	\cl_-[\phi]=0$ is nothing but \eqref{2.4}, while $\cl_+[\phi']=0$ is obtained from \eqref{2.4} by differentiation in the spatial variable. We have thus idenitifed at least one element in each $Ker(\cl_-), Ker(\cl_+)$. Clearly, per the standard Sturm-Lioville's theory for Schr\"odinger operators acting on periodic functions, it is possible that there might be up to one additional elements in each of $Ker(\cl_-), Ker(\cl_+)$.
In our example of the dnoidal waves,   this does not happen and indeed, it turns out that $\ker(\cl_-)=span[\phi], \ker(\cl_+)=span[\phi']$, see Section \ref{sec:2.3} above. 
Interestingly, and based on this information only, we can identify the kernel of the self-adjoint matrix operator $\ch$.

  \subsection{Determination of $Ker(\ch)$}
  Before we start with our analysis, let us recall that the domain of the operator $\ch$ is so that the last component is mean-zero. We have the following result.
  \begin{proposition}
  	\label{prop:19}
 The kernel of $\mathcal{H}$ is two-dimensional. 
  	In fact, $\ker(\ch)=span[\Psi_1, \Psi_2]$, where
  	$$
  	\Psi_1=\begin{pmatrix}  \phi \\ 0\\0 \\ 0 \end{pmatrix}, \; \;
  	\Psi_2=\begin{pmatrix} 0\\ \phi' \\ \frac{1}{c^2-1}\phi \phi' \\ \frac{c}{c^2-1}\phi \phi'
  	\end{pmatrix}.
  	$$
  \end{proposition}
{\bf Remark:} Note that since the fourth component of both $\Psi_1, \Psi_2$ are mean-zero functions,  they do belong to the domain of $\ch$, as stated.
  \begin{proof}
  	 Let $\mathcal{H}\vec{f}=0$, where $\vec{f}=(f_1,f_2,f_3,f_4)$. Then, we have the following system
  	\begin{equation}\label{4.4}
  		\left\{ \begin{array}{ll}
  			\cl_-f_1=0\\
  			\cl_- f_2+\phi f_3=0\\
  			\phi f_2+f_3-cf_4=0\\
  			-cf_3+f_4=0.
  		\end{array} \right.
  	\end{equation}
  	
  	Obviously, we have
  	$$\Psi_1=\begin{pmatrix} \phi \\ 0 \\0 \\ 0 \end{pmatrix} \in \ker \mathcal{H}.$$
  	From the last three  equations of the system (\ref{4.4}), we have $f_4=cf_3$ and $\phi f_2=(c^2-1)f_3$. Plugging in the second equation of the system (\ref{4.4}), we get
  	$$
  	\left( \cl_- + \frac{\phi^2}{c^2-1}\right)f_2=0.
  	$$
  	Noting that $\cl_+=\cl_- + \frac{\phi^2}{c^2-1}$, the last equation means that $\cl_+f_2=0$.
  	Hence, and up to a constant, $f_2=\phi'$, $f_3=\frac{1}{c^2-1}\phi \phi'$,  $f_4=\frac{c}{c^2-1}\phi \phi'$, and with this we get
  	$$\Psi_2=\begin{pmatrix} 0\\ \phi' \\ \frac{1}{c^2-1}\phi \phi' \\ \frac{c}{c^2-1}\phi \phi' \end{pmatrix} \in \ker \mathcal{H}.
  	$$
  	Clearly, this describes all the linearly independent elemenst in $\ker(\ch)$, and Proposition \ref{prop:19} is established in full.
  \end{proof}
  Our next task is to identify $gKer(\cj\ch)\ominus \ker(\ch)$, as any basis of this subspace is relevant in our stability calculation, i.e. the matrix $\cd$, see Section \ref{sec:2.1}.
  \subsection{Identifying $\ker(\cj\ch)$}
  We start with the elements in $\ker(\cj\ch)\ominus \ker(\ch)$. That is, we would like to find elements $\vec{\eta}$, so that
  \begin{equation}
  	\label{c:20}
  	 \ch \vec{\eta}\in \ker(\cj)=span[(0,0,1,0), (0,0,0,1)].
  \end{equation}
  Let
  $\mathcal{H}\vec{\eta}_1=(0,0,1,0)$,
  where $\vec{\eta}_1=(\eta_{11}, \eta_{12},\eta_{13},\eta_{14})$. We obtain the following system
  $$\left\{ \begin{array}{ll}
  	\cl_-\eta_{11}=0\\
  	\cl_-\eta_{12}+\phi \eta_{13}=0 \\
  	\phi\eta_{12}+\eta_{13}-c\eta_{14}=1\\
  	-c\eta_{13}+\eta_{14}=0.
  \end{array} \right.
  $$
which, apart from $\ker(\cl_-)$,  has the following solution
  $$
  \vec{\eta}_1=\left( \begin{array}{cc}  0\\-\frac{1}{1-c^2}\cl_+^{-1}\phi \\ \frac{1}{1-c^2}+\frac{1}{(1-c^2)^2}\phi \cl_+^{-1}\phi \\ \frac{c}{1-c^2}+\frac{c}{(1-c^2)^2}\phi \cl_+^{-1}\phi \end{array} \right)
  $$
  Similarly,   let
  $
  \mathcal{H}\vec{\eta}_2=(0,0,0,1),
  $
  where $\vec{\eta}_2=(\eta_{21}, \eta_{22},\eta_{23},\eta_{24})$.  That is, we need to solve the following system
  $$\left\{ \begin{array}{ll}
  	\cl_-\eta_{21}=0\\
  	\cl_-\eta_{22}+\phi \eta_{23}=0 \\
  	\phi\eta_{22}+\eta_{23}-c\eta_{24}=0\\
  	-c\eta_{23}+\eta_{24}=1.
  \end{array} \right.
  $$
  We have the following solution
  $$
 \vec{\eta}_2= \left( \begin{array}{cc}  0\\-\frac{c}{1-c^2}\cl_+^{-1}\phi \\ \frac{c}{1-c^2}+\frac{c}{(1-c^2)^2}\phi \cl_+^{-1}\phi \\ \frac{1}{1-c^2}+\frac{c^2}{(1-c^2)^2}\phi \cl_+^{-1}\phi \end{array} \right)
  $$
  We have found two linearly independent solutions of \eqref{c:20}. 
  However, it is clear that the fourth components of both cannot be mean-value zero. In fact, in the generic case, we need to take a specific linear combination, so that we can achieve mean-value zero in the last component.

  To this end, 
  we take the following  linear combination,
  \begin{equation}
  	\label{c:30}
  	\vec{\tilde{\eta}}_1:= -\vec{\eta}_1 \left(\langle 1,1\rangle+\frac{c^2}{1-c^2}\langle \mathcal{L}_+^{-1}\phi, \phi\rangle\right)+ \vec{\eta}_2 \left(c\langle 1,1\rangle+\frac{c}{1-c^2}\langle \mathcal{L}_+^{-1}\phi, \phi\rangle\right)=
  \end{equation}
  $$=	-\vec{\eta}_1 \left(2T+\frac{c^2}{1-c^2}\langle \mathcal{L}_+^{-1}\phi, \phi\rangle\right)+ \vec{\eta}_2 \left(2Tc+\frac{c}{1-c^2}\langle \mathcal{L}_+^{-1}\phi, \phi\rangle\right). $$
  which clearly has the property $ \int_{-T}^T \vec{\tilde{\eta}}_{1 4} (x) dx =0$ and as a linear combination of
$\vec{\eta}_1, \vec{\eta}_2$ belongs to the set described in \eqref{c:20}, i.e. $	\vec{\tilde{\eta}}_1 \in \ker(\cj\ch)\ominus \ker(\ch)$. 
  We have thus established the following proposition.
  \begin{proposition}
  	\label{prop:27}
  	Under the assumption, $\ker(\cl_-)=span[\phi],\ \  \ker(\cl_+)=span[\phi']$, the kernel \\ $\ker(\cj\ch)$ is three  dimensional. More specifically,
  	$$
  	\ker(\cj\ch)=\ker(\ch)\oplus \left(\ker(\cj\ch)\ominus \ker(\ch) \right)= span[\Psi_1, \Psi_2]\oplus
  	span[\vec{\tilde{\eta}}_1].
  	$$
  \end{proposition}
 \noindent  Our next task is to describe the generalized kernel of the linearized oprator.
  \subsection{Structure of $gKer(\cj\ch)$}
  Since we have determined $\ker(\cj\ch)$ in Proposition \ref{prop:27}, it remains to find the generalized eigenvectors associated with this system.

  As a first step, we show that a linear combination of $\vec{\eta}_1, \vec{\eta}_2$ does not give rise to any adjoint eigenvectors.
  \begin{proposition}
  	\label{prop:34}
  	Assume that $\dpr{\cl_{+}^{-1} \phi}{\phi}\neq 0$. Then, the equation
  	\begin{equation}
  		\label{l:10}
  			\cj \ch \vec{f}=\gamma_1 \vec{\eta}_1+\gamma_2 \vec{\eta}_2.
  	\end{equation}
  	does not have any solutions, unless $\gamma_1=\gamma_2=0$.
  \end{proposition}
  \begin{proof}
  	Note that
   $$
   \gamma_1\vec{\eta}_1+\gamma_2\vec{\eta}_2=(\gamma_1+c\gamma_2)\vec{\eta}_1+\gamma_2\left( \begin{array}{c} 0\\ 0\\0\\1\end{array} \right).
   $$
  Assuming that \eqref{l:10} has solution $f$, we test it first against the vector $\vec{g}:=\left(\begin{array}{c}
  		0 \\ \phi \\ 0 \\ 0
  	\end{array}\right)$. We obtain
  $$
  -\f{\gamma_1+c\gamma_2}{1-c^2}\dpr{\cl_+^{-1} \phi}{\phi}=\dpr{\cj \ch \vec{f}}{\vec{g}}=
  \dpr{\ch \vec{f}}{\cj^*\vec{g}}=\dpr{\cl_- f_1}{\phi}=\dpr{f_1}{\cl_-\phi}=0.
  $$
  This yields that $\gamma_1+c\gamma_2=0$. But then, \eqref{l:10} turns into
  $$
  \cj \ch \vec{f}=\ga_2 \left( \begin{array}{c} 0\\ 0\\0\\1\end{array} \right).
  $$
As the last component of the left hand side is an exact derivative,  this  clearly does not have any solutions, unless $\ga_2=0$.
  \end{proof}
  Next, we show that the equations $\cj\ch f= \Psi_1$, $\cj\ch f= \Psi_2$ do have solutions, and we describe them in details.
To this end, let
  $\mathcal{J}\mathcal{H}\vec{\eta}_3=\Psi_1,$
  where $\vec{\eta}_3=(\eta_{31}, \eta_{32},\eta_{33},\eta_{34})$. We, have the following system
  $$\left\{ \begin{array}{ll}
  	\cl_-\eta_{31}=0\\
  	\cl_-\eta_{32}+\phi \eta_{33}=-\phi \\
  	-(-c\eta_{33}+\eta_{34})_x=0\\
  	-(\phi\eta_{32}+\eta_{33}-c\eta_{34})_x=0.
  \end{array} \right.
  $$
  Note that the first equation solves to $\eta_{3 1}=const. \phi$. But this  will not contribute anything to $gKer(\cj\ch)$ as this solution is already accounted for in $\ker(\ch)$.
  Thus, we take $\eta_{31}=0$. 
  
  Integrating the last two equations yields,    new integration constants.
  $$\left\{ \begin{array}{ll}
  	\cl_-\eta_{32}+\phi \eta_{33}=-\phi \\
  	c\eta_{33}-\eta_{34}=d_1\\
  	-\phi\eta_{32}-\eta_{33}+c\eta_{34}=d_2,
  \end{array} \right.
  $$
  From these three equations, we determine
  $$
  \vec{\eta}_3=\left(\begin{array}{c}
  0 \\ -k\cl_+^{-1}\phi \\ \f{k}{1-c^2} \phi \cl_+^{-1} \phi +k_1\\ \f{c k}{1-c^2} \phi \cl_+^{-1} \phi+k_2
  \end{array}\right)
  $$
  where the particular form of the constants $k, k_1, k_2$ depends on $c, d_1, d_2$, but  is otherwise  unimportant in our analysis. Note that
  $$
   \vec{\eta}_3=k(1-c^2) \vec{\eta}_1 + \left(\begin{array}{c}
   	0 \\ 0  \\ \tilde{k}_1\\ \tilde{k}_2
   \end{array}\right) = k(1-c^2) \vec{\eta}_1+\tilde{k}_2 \vec{\tilde{\eta}}_2+ \tilde{k}_1 \left(\begin{array}{c}
   0 \\ 0  \\ 1\\ 0
\end{array}\right).
   $$
  where again $\tilde{k}_1, \tilde{k}_2$ are two constants. As $\vec{\eta}_1, \vec{\tilde{\eta}}_2\in \ker(\cj\ch)$, we can clearly take
  $\vec{\tilde{\eta}}_3:=\left(\begin{array}{c}
  		0 \\ 0  \\ 1\\ 0
  	\end{array}\right)$. In fact, a direct verification confirms that
  $\cj\ch \vec{\tilde{\eta}}_3=-\Psi_1$. Note that as $\vec{\tilde{\eta}}_{3 4}=0$, it is mean-value zero in the last component. So it is acceptable vector in our analysis.

Lastly, we consider the equation $\cj\ch \eta_4=\Psi_2$. We have the system
  $$\left\{ \begin{array}{ll}
  	\cl_-\eta_{41}=\phi'\\
  	\cl_-\eta_{32}+\phi \eta_{33}=0\\
  	-(\phi \eta_{4 2}+\eta_{43}-c \eta_{44})_x=-\f{c}{2(1-c^2)} (\phi^2)_x\\
  	 -(-c\eta_{4 3}+ \eta_{4 4})_x=-\f{1}{2(1-c^2)}(\phi^2)_x
  \end{array} \right.
  $$
  Note that due to the presence of the vectors $\vec{\tilde{\eta}}_2, \vec{\tilde{\eta}}_3$ in our system, we can integrate the last two equations with constants of integration zero  - if non-zero, their contribution can be written in terms of $span[\vec{\tilde{\eta}}_2, \vec{\tilde{\eta}}_3]$ and hence safely ignored.  Thus, we find the solution
  $$
  \vec{\eta}_4 = \left(\begin{array}{c}
  	\cl_-^{-1} [\phi']\\  -\f{c}{(1-c^2)^2} \cl_+^{-1} [\phi^3] \\ \f{c}{(1-c^2)^2}\phi^2 +\f{c}{(1-c^2)^3} \phi \cl_+^{-1} [\phi^3] \\ \f{c^2+1}{2(1-c^2)^2} \phi^2 + \f{c^2}{(1-c^2)^3} \phi \cl_+^{-1} [\phi^3]
  \end{array}\right)
  $$
  Using that $\cl_+\phi=--\frac{1}{1-c^2}\phi^3$, we get $\cl_+^{-1}\phi^3=-(1-c^2)\phi$. This leads to the following representation
  $$\vec{\eta}_4 = \left(\begin{array}{c}
  	\cl_-^{-1} [\phi']\\  \f{c}{1-c^2} \phi \\ 0 \\ \f{1}{2(1-c^2)} \phi^2.
  \end{array}\right).
  $$
  Here, it becomes clear that $\vec{\eta}_{4 4}$ does not have mean-value zero, hence the solution $\vec{\eta}_4$ does not belong to the required $D(\cj \ch)$.
 We formulate our result in the following proposition.
 \begin{proposition}
 	\label{prop:36}
 	Assume that $\dpr{\cl_+^{-1} \phi}{\phi}\neq 0$. Then there is only one  linearly independent first generation adjoint eigenvector, namely  $\vec{\tilde{\eta}}_3$. That is
 	$$
 	\ker((\cj\ch)^2)\ominus \ker(\cj\ch) = span[\vec{\tilde{\eta}}_3].
 	$$
 \end{proposition}
\begin{proof}
	Our strategy here is as follows - we will show that
	$$
	\ker((\cj\ch)^2)\ominus \ker(\cj\ch) = span[\vec{\tilde{\eta}}_3, \vec{\eta}_4],
	$$
	if we do not restrict with the condition that $\int_{-T}^T \vec{\eta}_{4 4}(x) dx=0$. Then, due to this restriction, $\vec{\eta}_4$ is excluded, whence our claim follows.

	We have  so far  proved that\footnote{ when considered unrestricted vectors $\vec{\eta}$, that is not-necessarily $\int_{-T}^T \vec{\eta}_{ 4 4 }dx=0$},  $\ker((\cj\ch)^2)\ominus \ker(\cj\ch) \supseteq span[\vec{\tilde{\eta}}_3, \vec{\eta}_4]$. In order to show the other direction, set the equation $\cj\ch f\in \ker(\cj\ch)$ or
	$$
	\cj\ch f= \la_1 \Psi_1+\la_2 \Psi_2+\la_3 \eta_1+\la_4 \eta_2.
	$$
	As $\cj\ch\vec{\tilde{\eta}}_3=-\Psi_1, \cj\ch \vec{\eta}_4=\Psi_2$, we have
	$$
	\cj\ch [f+\la_1  \vec{\tilde{\eta}}_3-\la_2  \vec{\eta}_4]=\la_3 \eta_1+\la_4 \eta_2.
	$$
	We now apply Proposition \ref{prop:34}, to conclude that $\la_3=\la_4=0$. It follows that
	$f-\la_1  \vec{\tilde{\eta}}_3 -\la_2  \vec{\eta}_4\in\ker 	\cj\ch$, which is the claim.
\end{proof}
Finally, we consider the possibility of second adjoint eigenvectors.
Expectedly, it turns out that there are not any.
\begin{proposition}
	\label{prop:42}
	We have,
	$$
	\ker((\cj\ch)^3)\ominus \ker((\cj\ch)^2)= \{0\}.
	$$
\end{proposition}
  \begin{proof}
  	Consider the subspace $	\ker((\cj\ch)^3)\ominus \ker((\cj\ch)^2)$. That is, set up an equation
  	$$
  	\cj \ch f= \mu \vec{\tilde{\eta}}_3
  	$$
  	for some scalar $\mu$.
  	We test it against $ \vec{\tilde{\eta}}_3$. We obtain
  $$
  \mu =\dpr{\cj \ch f}{\vec{\tilde{\eta}}_3}=-
  \dpr{ \ch f}{\cj \vec{\tilde{\eta}}_3}=0,
  $$
  as $\cj \vec{\tilde{\eta}}_3=0$.
  This implies $\mu=0$, which establishes the claim.
  \end{proof}
\noindent We collect all the findings of this section in the following proposition.
\begin{proposition}
	\label{prop:50}
	Suppose that  $\dpr{\cl_+^{-1} \phi}{\phi}\neq 0$.  Then,
	$$
	gKer(\cj\ch)\ominus \ker(\ch)=span[ \vec{\tilde{\eta}}_1, \vec{\tilde{\eta}}_3].
	$$
\end{proposition}

\subsection{The Morse index of $\ch$}
We have the following proposition. 
\begin{proposition}
	\label{prop:17}
	For the solution $\phi$  given by \eqref{4.1},  we have that  $n(\ch)\leq 1$.	
\end{proposition}
{\bf Remark:} We later easily establish that in fact $n(\ch)=1$, as a consequence of the index counting formula \eqref{b:20}. However, and again from the same formula, we have the spectral stability of the waves in the case, when $n(\ch)=0$. 
\begin{proof}
 We begin with the observation that due to the tensorial structure of $\ch=\cl_-\otimes\ch_1$, where
	$$
	\ch_1=\begin{pmatrix} \cl_- & \phi &0 \\ \phi &1&-c \\ 0&-c&1 \end{pmatrix}.
$$
Clearly, we have that $n(\ch)=n(\ch_1)+n(\cl_-)$.
By Proposition \ref{prop:5}, we have  that  $\cl_-\geq 0$. Thus, it remains to prove  that $n(\ch_1)\leq 1$. 
To this end,
consider the    quadratic form associated to $\ch_1$, namely
$$
q_1(u,v,w)=\langle \cl_- u,u\rangle+2\langle \phi u, v\rangle+\langle v,v\rangle-2c\langle v,w\rangle+\langle w,w\rangle.
$$
	By Cauchy-Schwartz inequality
	\begin{equation}
		\label{q:20}
			\langle v,v\rangle-2c\langle v,w\rangle+\langle w,w\rangle\geq (1-c^2)\langle v,v\rangle,
	\end{equation}
	whence, $q_1(u,v,w)\geq q_2(v,w)=\dpr{\cl_ - u}{u}+2\langle \phi u, v\rangle+(1-c^2)\langle v,v\rangle$. We further estimate $q_2$ as follows
	\begin{equation}\label{q:20a}
\begin{array}{ll}
		q_2(u,v)&= \langle \cl_- u,u\rangle+2\langle \phi u,v\rangle+(1-c^2)\langle v,v\rangle \\
		&= \langle \cl_- u,u\rangle+\int_{-T}^{T}{\left( \sqrt{1-c^2}v+\frac{\phi}{\sqrt{1-c^2}}u\right)^2 }dx-\frac{1}{1-c^2}\int_{-T}^{T}{\phi^2u^2}dx\\
		&\geq  \langle  \left(\cl_- -\frac{\phi^2}{1-c^2}\right)u,u\rangle =\langle \cl_+ u,u\rangle.
	\end{array}
\end{equation}

	It is now easy to conlcude that $n(\ch_1)\leq 1$. Indeed, based on the fact that $n(\cl_+)=1$, taking $u$ orthogonal to the ground state\footnote{this function is in fact explicitly known in both cases} of $\cl_+$ guarantees that $\langle \cl_+ u,u\rangle\geq 0$. Thus,
	$$
	\inf_{u\perp \textup{ground state of }\ \cl_+} q_1(u,v,w)\geq \inf_{u\perp \textup{ground state of }\ \cl_+} q_2(u,v)\geq \inf_{u\perp \textup{ground state of }\ \cl_+} \langle \cl_+ u,u\rangle\geq 0.
	$$
	This establishes $n(\ch_1)\leq 1$. 

\end{proof}

  \section{Proof of Theorem \ref{thm1}}
  \label{sec:4}
 We apply the instability index theory, as developed in Section \ref{sec:2}. To this end, we need to identify first the spaces $X, X^*, H$. To this end, introduce $X=(H^1_{per.}[-T,T])^2\times (L^2[-T,T])^2$, with 
  $X^*=(H^{-1}_{per.}[-T,T])^2\times (L^2[-T,T])^2$, while clearly we use the base Hilbert space $H=(L^2[-T,T])^4$. Clearly, $\ch:X\to X^*$ is a bounded and symmetric, while $\cj:D(\cj)=(H^1[-T,T])^4\to X$ is a closed operator.

  By Proposition \ref{prop:50}, $\cd\in \cm_{2\times 2}$, and we take on computing its elements. A direct calculation yields
  $$
  \ch \vec{\tilde{\eta}}_3 = \left(\begin{array}{c} 0 \\ \phi \\ 1 \\ -c\end{array}\right).
  $$
  Next, we compute  $\ch\vec{\tilde{\eta}}_1$, by using  the formula \eqref{c:30}, and   the values
  $
  \ch \vec{\eta}_1=\left(\begin{array}{c} 0 \\ 0 \\ 1 \\ 0\end{array}\right),
  \ch \vec{\eta}_2 = \left(\begin{array}{c} 0 \\ 0 \\ 0 \\ 1\end{array}\right),
  $
  which we have by construction. 
  We obtain
  $$
  \ch\vec{\tilde{\eta}}_1 = \left(\begin{array}{c} 0 \\ 0 \\ -\left(\langle 1,1\rangle+\frac{c^2}{1-c^2}\langle \mathcal{L}_+^{-1}\phi, \phi\rangle \right)\\ \left(c\langle 1,1\rangle+\frac{c}{1-c^2}\langle \mathcal{L}_+^{-1}\phi, \phi\rangle\right) \end{array}\right), 
  $$
  After some algebraic manipulations (note $\dpr{1}{1}=2T$),  we obtain the matrix $\cd$ in the following form
  $$\begin{array}{ll}
  	\cd_{11}=\dpr{\ch \vec{\tilde{\eta}}_1}{\vec{\tilde{\eta}}_1}=
  2 T \left( 2 T+\frac{c^2}{1-c^2}\langle \mathcal{L}_+^{-1}\phi, \phi\rangle \right) 
  \left( 2T +\frac{1}{1-c^2}\langle \mathcal{L}_+^{-1}\phi, \phi\rangle\right) \\
  	\\
  	\cd_{12}=\cd_{21}=\dpr{\ch \vec{\tilde{\eta}}_1}{\vec{\tilde{\eta}}_3}= - 2 T 
  	\left(2T+\frac{c^2}{1-c^2}\langle \mathcal{L}_+^{-1}\phi, \phi\rangle \right) \\
  	\\
  	\cd_{22}=\dpr{\ch \vec{\tilde{\eta}}_3}{\vec{\tilde{\eta}}_3}= 2 T.
  \end{array}
  $$
    It follows that
  $$
  \det(\cd)= 4 T^2 \dpr{\cl_+^{-1} \phi}{\phi} \left(2 T +\f{c^2}{1-c^2}\dpr{\mathcal{L}_+^{-1}\phi}{\phi}\right)
  $$
 Recall that we have established that $n(\ch)\leq 1$, see Proposition \ref{prop:17}.  Due to the formula \eqref{170}, it must be that $n(\cd)\leq 1)$. Thus, the spectral stability is equivalent to the property $n(\cd)=1$. In fact, due to the formula \eqref{b:20}, it follows that $n(\ch)\geq n(\cd)$, so establishing $n(\cd)=1$ implies $n(\ch)=1$ as well, as announced earlier. 
 
Next,  in order to prove $n(\cd)=1$,  we shall need to show that $\det(\cd)<0$, as this gurantees that $\cd$ has a negative eigenvalue.
 To this end, recall that we have  $\mathcal{L}_+\phi'=0 $. 
               The function
      $$\Phi(x)=\phi'(x)\int^{x}{\frac{1}{\phi'^2(s)}}ds, \; \; \left| \begin{array}{cc} \phi'& \Phi \\ \phi'' & \Phi'\end{array}\right|=1 $$
      is also solution of $ L\Phi=0 $.
      Formally, since $\Phi'$ has   zeros  using the identities
        $$
        \frac{1}{sn^2(y,\kappa)}=-\frac{1}{dn(y, \kappa)}\frac{\partial}{\partial_y}\frac{cn(x, \kappa)}{sn(y, \kappa)}$$
        and integrating by parts we get
       $$
        \Phi(x) =\frac{1}{\alpha^2\kappa^2\varphi_0}\left[\frac{1-2sn^2(\alpha x, \kappa )}{dn (\alpha x, \kappa) }
        - \alpha \kappa^2sn(\alpha x, \kappa)cn(\alpha x, \kappa)  \int_{0}^{x}{\frac{1-2sn^2(\alpha s, \kappa )}{dn^2(\alpha s, \kappa)}}ds\right].
       $$
       Thus, we may construct Green function
         $$\mathcal{L}_+^{-1}f=\phi'\int_{0}^{x}{\Phi(s)f(s)}ds-\Phi(s)\int_{0}^{x}{\phi'(s)f(s)}s+C_f\Phi(x), $$
         where $C_f$ is chosen such that $\mathcal{L}_+^{-1}f$ is periodic with same period as $\phi(x)$.
  Integration by parts, we get
      $$ \langle \mathcal{L}_+^{-1}\phi,\phi \rangle=-\langle\phi^3, \Phi\rangle+\phi^2(T)\langle\phi,\Phi\rangle-\frac{\phi''(T)}{2\Phi'(T)}\langle\phi,\Phi\rangle^2.$$
         Similarly as in  \cite{DK}, using that $\phi^3=-(1-c^2)\mathcal{L}\phi$ and $\langle \Phi'',\phi\rangle=2\phi(T)\Phi'(T)+\dpr{\Phi}{\phi''}$, we have the following results,
      $$\langle \mathcal{L}_+^{-1}\phi,\phi \rangle =2(1-c^2)\phi(T)\Phi'(T)+\phi^2(T)\langle\phi,\Phi\rangle-\frac{\phi''(T)}{2\Phi'(T)}\langle\phi,\Phi\rangle^2.$$
      Integrating by parts, we get
      $$\Phi'(T)\phi(T)=\frac{1}{\alpha \kappa^2}[2(1-\kappa^2)K-(2-\kappa^2)E].$$
$$\langle \phi, \Phi\rangle=\frac{1}{\alpha^3 \kappa^2}[E(\kappa)-K(\kappa)].$$
$$\frac{\phi''(T)}{2\Phi'(T)}\langle \phi, \Phi\rangle=\frac{\phi_0^2\kappa^2(1-\kappa^2)[E-K]}{2[2(1-\kappa^2)K-(2-\kappa^2)E]}.$$
which leads to the following result
  \begin{equation}\label{DK1}
  \langle \mathcal{L}_+^{-1}\phi,\phi \rangle=-\frac{2(1-c^2)}{\alpha}\frac{E^2(\kappa)-(1-\kappa^2)K^2(\kappa)}{(2-\kappa^2)E(\kappa)-2(1-\kappa^2)K(\kappa)}<0.
  \end{equation}
      For the other term, we have 
      $$\begin{array}{ll} 2T+\frac{c^2}{1-c^2}\langle \mathcal{L}_+^{-1}\phi, \phi\rangle &=\frac{2}{\alpha}\left[ K(\kappa)-c^2\frac{E^2(\kappa)-(1-\kappa^2)K^2(\kappa)}{(2-\kappa^2)E(\kappa)-2(1-\kappa^2)K(\kappa)}\right]\\
      \\
      &=\frac{2}{\alpha}\frac{E^2(\kappa)-(1-\kappa^2)K^2(\kappa)}{(2-\kappa^2)E(\kappa)-2(1-\kappa^2)K(\kappa)}\left[ \frac{(2-\kappa^2)E(\kappa)K(\kappa)-2(1-\kappa^2)K^2(\kappa)}{E^2(\kappa)-(1-\kappa^2)K^2(\kappa)}-c^2\right]>0.
      \end{array}$$
since 
$$
\frac{(2-\kappa^2)E(\kappa)K(\kappa)-2(1-\kappa^2)K^2(\kappa)}{E^2(\kappa)-(1-\kappa^2)K^2(\kappa)}>1>c^2.
$$
All together, $\det(\cd)<0$ and the spectral stability is established. 
\\ 
\\ 
{\bf Acknowledgements:} 	Stanislavova is partially supported by the NSF, under award \# 2210867. Stefanov is partially supported by the NSF, under award \# 2204788.


\begin{thebibliography}{99}
	
	\bibitem{AlBoHe} J. Albert, J. Bona, D. Henry, \emph{Sufficient
	conditions for stability of solitary wave solutions of model
	equations for long waves}, {\em Physica D}, {\bf 24},  (1987), p. 343--366. 
	
	\bibitem{AlBo} J. Albert, J. Bona, \emph{Total positivity and
	stability of internal waves in stratified fluids of finite
	depth}, {\em IMA J. Appl. Math.}, {\bf 46}, (1991), p. 1--19. 
	
	
\bibitem{AB} J. Angulo, C. Brango, \emph{Orbital stability for the periodic Zakharov system}, {\em Nonlinearity}, {\bf 24}, (10), (2011), p. 2913--2932. 
	
	\bibitem{Be} T. B. Benjamin, \emph{The stability of solitary waves},{\em 
	Proc. R. Soc. Lond. Ser. A}, {\bf 338}, (1972),  p. 153--183. 
	
	\bibitem{Bo} J. Bona, \emph{On the stability theory of solitary
	waves}, {\em Proc. R. Soc. Lond. Ser. A}, {\bf 344}, (1975),  p. 363--374. 
	
	\bibitem{BoSoSt} J. Bona, P. Souganidis, W. Strauss, \emph{Stability
	and instability of solitary waves of KdV type},
	{\em Proc.Roy.Soc.London A}, {\bf 411}, (1987), p. 395--412
	
	\bibitem{ByFr} P. F. Byrd, M. D. Friedman, \textit{Handbook of
		Elliptic Integrals for Engineers and Scientists}, sec. ed.,
	Springer-Verlag, New York, (1971)

\bibitem{DK} B. Deconinck, T. Kapitula, \emph{On the spectral and orbital stability of spatially periodic stationary solutions of generalized Korteweg–de Vries equations}. Hamiltonian partial differential equations and applications, p. 285--322, Fields Inst. Commun., {\bf 75}, Fields Inst. Res. Math. Sci., Toronto, ON, 2015.
	
	\bibitem{GrShSt1} M. Grillakis, J. Shatah, W. Strauss, \emph{Stability
	theory of solitary waves in the presence of symmetry I}, {\em J.
		Funct. Anal.}, {\bf 74}, (1987), p. 160--197. 
	
	\bibitem{GrShSt2} M. Grillakis, J. Shatah, W. Strauss, \emph{Stability
	theory of solitary waves in the presence of symmetry II,} \em{J.
		Funct. Anal.}, {\bf 94}, (1990), p. 308--348. 
	
	\bibitem{GuSh} B. Guo, L. Shen,  \emph{The existence and uniqueness
	of the classical solutions on the periodic initial value
	problem for Zakharov equation}, {\em Acta Math.Appl.Sinica},
	{\bf 5}, (1982), p. 310--324. 
	
	\bibitem{KP} T. Kapitula, K. Promislow, Spectral and Dynamical Stability of
	Nonlinear Waves, 185, Applied Mathematical Sciences, 2013. 
	
	
	\bibitem{LZ} Z. Lin, C.  Zeng, {\emph Instability, index theorem, and exponential trichotomy for Linear Hamiltonian PDEs},  {\em Mem. Amer. Math. Soc.} {\bf  275},  (2022), no. 1347. 
	
	\bibitem{MaWi} W. Magnus, S. Winkler, \textit{Hill's Equation},
	Interscience, Tracts in Pure and Appl. Math. Wiley,NY.,20,1976
	
 
	\bibitem{We1} M. Weinstein, \emph{Lyapunov stability of ground states
	of nonlinear dispersive evolution equations}, {\em Comm.Pure Appl.
		Math.}, {\bf 39}, (1986), p. 2413--2422. 
	
	\bibitem{We2} M. Weinstein, \emph{ Modulation stability of ground
	states of nonlinear Schr\"odinger equations}, {\em SIAM J. Math. Anal.}, 
	{\bf 16}, (1985),  p. 472--490. 
	
	\bibitem{Wu} Y. Wu, \emph{Orbital stability of solitary waves of
	Zakharov system}, {\em J.Math.Phys.}, {\bf 35}(5), (1994), p. 2413--2422. 
	
	\bibitem{Za} V. E. Zakharov, {\emph Collapse of Langmuir waves}, {\em Soviet
		Phys. JETP}, {\bf 35}, (1972), p. 908--814. 
	
\end{thebibliography}
\end{document}